\pdfoutput=1
\pdfoutput=1
\pdfoutput=1
\pdfoutput=1
\pdfoutput=1
\documentclass[a4paper,12pt]{article}
\usepackage{calc}
\usepackage[all]{xy}
\usepackage[centertags]{amsmath}
\usepackage{latexsym}
\usepackage{amsfonts}
\usepackage{graphicx}
\usepackage{tikz}
\usepackage{cases}
\usepackage{amssymb}
\usepackage{amsthm}
\usepackage{color}
\usepackage{fancyhdr}
\usepackage[dvips]{epsfig}
\usepackage{newlfont}
\usepackage[latin1]{inputenc}
\usepackage[latin1]{inputenc}
\usepackage[english,french]{babel}
\usepackage{graphicx}
\usepackage{t1enc}
\usepackage[english,french]{babel}
\usepackage{fancybox}
\usepackage{graphicx}
\usepackage{t1enc}
\usepackage[french2]{minitoc}
\usepackage{mathrsfs}
\usepackage{amsfonts}
\usepackage{wasysym}
\usepackage{hyperref}
\allowdisplaybreaks
\usepackage{float}

\fancypagestyle{plain}{ \fancyhead{}

	\rhead{\textbf{}} \rfoot{\footnotesize{\textsf{\tiny }}}
	\cfoot{\footnotesize{\textbf{ \thepage}}}} \hfuzz2pt
\newlength{\defbaselineskip}
\numberwithin{equation}{section} 
\newtheorem{theorem}{Theorem}[section]
\newtheorem{corollary}[theorem]{Corollary}

\newtheorem{lemma}[theorem]{Lemma}
\newtheorem{proposition}[theorem]{Proposition}

\newtheorem{remark}{Remark}[section]

\newcommand{\Z}{{\mathbb Z}}

 \makeatletter
\newcommand{\thechapterwords}
{ \ifcase \thechapter\or 1\or 2\or 3\or 4\or 5\or
	6\or 7\or 8\or 9\or 10\or 11\fi}
\def\thickhrulefill{\leavevmode \leaders \hrule height 2ex \hfill \kern \z@}
\def\@makechapterhead#1{%
	\vspace*{15\p@}%
	{\parindent \z@ \centering \reset@font
		\thickhrulefill\quad
		\scshape  {\chapnumfont \@chapapp{}}{\chapnumfont \thechapterwords}
		\quad \thickhrulefill
		\par\nobreak
		\vspace*{15\p@}%
		\interlinepenalty\@M
		\hrule
		\vspace*{15\p@}%
		\huge {\bfseries  #1}\par\nobreak
		\par
		\vspace*{15\p@}%
		\hrule
		\vskip 15\p@
	}}
	\def\@makeschapterhead#1{%
		\vspace*{15\p@}%
		{\parindent \z@ \centering \reset@font
			\thickhrulefill
			\par\nobreak
			\vspace*{15\p@}%
			\interlinepenalty\@M
			\hrule
			\vspace*{15\p@}%
			\Huge \bfseries #1\par\nobreak
			\par
			\vspace*{15\p@}%
			\hrule
			\vskip 30\p@
		}}
		
		\DeclareFixedFont{\chapnumfont}{T1}{phv}{b}{n}{20pt}
		\DeclareFixedFont{\chapchapfont}{T1}{phv}{b}{n}{16pt}
		\DeclareFixedFont{\chaptitfont}{T1}{phv}{b}{n}{24.88pt}
		\def\@makechapterhead#1{%
			\vspace*{15\p@}%
			{\parindent \z@ \centering \reset@font
				\thickhrulefill\quad
				\scshape {\chaptitfont\color[rgb]{0.00,0.50,1.00}\@chapapp{}}
				{\chapnumfont \thechapterwords}
				\quad \thickhrulefill
				\par\nobreak
				\vspace*{15\p@}%
				\interlinepenalty\@M
				\hrule
				\vspace*{15\p@}%
				{\Large\bfseries #1}\par\nobreak
				\par
				\vspace*{15\p@}%
				\hrule
				\vskip 30\p@
			}}%
			
\begin{document}
				\title{Dynamic of Pair of some Distributions: Bi-lagrangian structure and its prolongations on the (co)tangent bundles, and Cherry flow}
				\date{ }
				\author{ \large{Bertuel  Tangue NDAWA }\\
\normalsize{bertuelt@yahoo.fr}
					\vspace{0.5cm}\\\today}
				\maketitle
				
				\selectlanguage{english}
	In the memory of Dad NDAWA Joseph, Dad TCHOUATEUN Pierre, and Dad PAHANE Isaac.
	\section*{Abstract}
				We consider a bi-Lagrangian manifold $(M,\omega,\mathcal{F}_{1},\mathcal{F}_{2})$. That is,  $\omega$ is a 2-form, closed and non-degenerate (called symplectic form) on  $M$, and  $(\mathcal{F}_{1},\mathcal{F}_{2})$ is a pair of transversal Lagrangian foliations on the symplectic manifold $(M,\omega)$. In this case, $(\omega, \mathcal{F}_{1},\mathcal{F}_{2})$ is a bi-Lagrangian structure on $M$.
				
				In this paper, we prolong a bi-Lagrangian structure on $M$  on its tangent bundle $TM$ and  its cotangent bundle $T^{*}M$ in different ways. As a consequence some dynamics on the bi-Lagrangian structure of $M$ can be prolonged as  dynamics on the bi-Lagrangian structure of $TM$ and  $T^{*}M$. 
				 Observe that   a pair of transversal vector fields without singularity on the 2-torus $\mathbb{T}^2=\mathbb{S}^1\times\mathbb{S}^1$ endowed with a symplectic form defines a bi-Lagrangian structure on $\mathbb{T}^2$. This sparked our curiosity. By studying the dynamic of pairs of vector fields on $\mathbb{T}^2$,  we found that some circle maps with a flat piece (called Cherry maps) can be generated by a pair of vector fields.  Moreover, the  push forward action of the set of diffeomorphisms  $\mathbb{T}^2$  on the set of its vector fields induces
				a conjugation action on the set of generated Cherry maps.  
				
				\textbf{Keywords}: Symplectic, Symplectomorphism, Bi-Lagrangian, Para-K\"{a}hler, Hess connection, Cherry vector field, Circle map, Flat piece.
				
				\textbf{MSC2010}: 53D05, 53D12.
				
				\textbf{Acknowledgment}:




\section{Introduction}

The bi-Lagrangian manifolds has been intensively explored in the past years, see \cite{2, 11, 3, GB1, FR1, FR2, 7, FE}. Among the many reasons to study them, they are the areas of geometric quantization (see \cite{7}) and of Koszul-Vinberg Cohomology (see \cite{GB1}). The data of a  bi-Lagrangian  manifold $(M,\omega,\mathcal{F}_{1},\mathcal{F}_{2})$  induces in a one-to-one way a para-K\"{a}hler structure $(G,F)$  on $M$ (that is, $G$ is a pseudo-Riemannian metric  on $M$ and $F$ is a para-complex structure  on $M$ which permutes with $G$ in the following sens: $G(F(\cdot),F(\cdot))=-G(\cdot,\cdot)$). The three tensors $\omega$, $G$ and $F$  are linked by the relation:  $\omega(\cdot,\cdot)=G(F(\cdot),\cdot)$, see \cite{FR1, FR2, 1, FE}. Therefore, a bi-Lagrangian manifold put together symplectic, semi-Riemannian, and almost product structures (note that, if $(M,\omega,\mathcal{F}_{1},\mathcal{F}_{2})$ is a bi-Lagrangian manifold, then $(\mathcal{F}_{1},\mathcal{F}_{2})$ is a bi-Lagrangian structure on the symplectic manifold $(M,\omega)$). In  \cite{7}, the author proves that a bi-Lagrangian  manifold $(M,\omega,\mathcal{F}_{1},\mathcal{F}_{2})$ have a unique torsion free (torsionless) connection $\nabla$ called Hess (or bi-Lagrangian) connection  which parallelizes $\omega$ ($\nabla$ is a symplectic connection) and preserves both foliations. The explicit formula of a Hess connection was given in \cite{2, 11, 3}.  In the case of 2-torus $\mathbb{T}^2$, a Lagrangian foliation can be defined by a vector field without singularity. Thus the study of Lagrangian foliation on the 2-torus is a complement to the study of the dynamic of some vector fields on $\mathbb{T}^2$ with a finite number of singularities, see \cite{TC, MSM, PM, LP1}. 

This paper extends on the one hand, the results in \cite{TNB2}, and on the other hand, some  results in \cite{MSM, LP1}. 

Before we can explain more precisely and prove our results, it is necessary to present some definitions, fix some notations   and formulate some known results we need.

\subsection{Basics definitions, properties and notations}\label{sub1}
We assume that all the objects are smooth throughout this paper unless otherwise stated.

\subsubsection{Foliation}
Let $M$ be an $m$-manifold. We define a $k$-dimensional foliation $\mathcal{F}$ on $M$ to be a decomposition of $M$ into a union of disjoint, non-empty, connected, immersed $k$-dimensional submanifolds  $\{S_x\}_{x\in M}$, called the leaves of the foliation, with the following property (called completely integrable property): every point $y$ in $M$ has a coordinate chart   $(U, y^1,\dots, y^m)$ such that for each leaf $S_x$ the components of $U\cap S_x$  are described by the equations $y^{k+1}\mbox{=constant},\dots, y^m=\mbox{constant}$, see \cite[p. 370]{law}, \cite[p. 501]{lee}.

The expressions $T\mathcal{F}\subset TM$ and $\Gamma\left(T\mathcal{F}\right)=\Gamma\left(\mathcal{F}\right)\subset \Gamma\left(TM\right)= \mathfrak{X}(M)$ (the set of vector fields on $M$) denote the tangent bundle of $\mathcal{F}$ and the set of sections of $T\mathcal{F}$  respectively. For each point $y\in M$ the vector subspace $T_yS_y\subset T_yM$ is called the bundle tangent of $\mathcal{F}$ over $y$ and is denoted by $\mathcal{F}_y$ or $T_y\mathcal{F}$. 

The Lie bracket of two vector fields $X$, $Y$ is that defined by
 $[X,Y]:=X\circ Y-Y\circ X$.

Note that the completely integrable property of a foliation $\mathcal{F}$ means that $\Gamma\left(\mathcal{F}\right)$ is stable under the Lie bracket; that is, if $X,Y\in \Gamma\left(\mathcal{F}\right)$ then $[X,Y]\in\Gamma\left(\mathcal{F}\right)$. This is the Frobenius Theorem, see \cite[p. 496]{lee}.

Let $\psi : M\longrightarrow N$ be a diffeomorphism. The  push forward  $\psi_*\mathcal{F}=\{\psi (S_x)\}_{x\in M}$  of $\mathcal{F}$ by $\psi$ is a foliation on $N$, and
\begin{equation}
\Gamma\left(\psi_*\mathcal{F}\right):=\{\psi_*X,\; X \in\Gamma\left(\mathcal{F}\right)\}=\psi_*\Gamma\left(\mathcal{F}\right).\label{Bieq2}
\end{equation}

We denote by $Diff(M)$ the set of diffeomorphisms from $M$ to itself. 

Let $S\subset M$ be a submanifold. The conormal  space at a point $x\in S$ is defined by
$$N^*_xS=\{\xi_x\in T_xM,\;{\xi_x}_{|T_xS}=0\},$$
and the conormal bundle of $S$ \textsl{}is 
$$N^*S=\{(\xi_x,x)\in T^*M,\;\xi_x\in N^*_xS\}.$$
The conormal bundle of a foliation $\mathcal{F}=\{S_x\}_{x\in M}$ is 
$$N^*\mathcal{F}=\{N^*S_x\}_{x\in M}.$$

If the manifold $M$ is endowed with a symplectic form $\omega$ (as a consequence, $m=2n$),  a submanifold  $S\subset M$ is Lagrangian if for every $X\in\Gamma\left(TS\right)$, $\omega(X,Y)=0$ if and only if $Y\in\Gamma\left( TS\right).$
That is, the orthogonal section
$$\Gamma\left(TS\right)^{\perp}=\left\{X\in \mathfrak{X}(M):\; \omega(X,Y)=0, \;Y\in \Gamma\left(TS\right)\right\}$$
of  $\Gamma\left(TS\right)$ is equal to $\Gamma\left(TS\right)$.

A foliation $\mathcal{F}$ is Lagrangian if its leaves are lagrangian. That is, 
$\Gamma\left(\mathcal{F}\right)^{\perp}=\Gamma\left(\mathcal{F}\right).$

 A bi-Lagrangian structure on  $M$ consists of a pair $(\mathcal{F}_{1},\mathcal{F}_{2})$ of transversal Lagrangian foliations together with a symplectic form $\omega$. As a consequence, $TM=T\mathcal{F}_{1}\oplus T\mathcal{F}_{2}$.

Let $(\mathcal{F}_{1},\mathcal{F}_{2})$ be a bi-Lagrangian structure on a symplectic $2n$-manifold
$(M,\omega)$. Every point in $M$ has an open neighborhood $U$ which is the domain of a chart whose local coordinates $(p^1,\dots,p^{n},q^1,\dots,q^{n})$ are such that
\begin{equation*}
\begin{cases}
\Gamma(\mathcal{F}_1)_{\mid U}=\left<\frac{\partial}{\partial
	p^1},\dots,\frac{\partial}{\partial p^n}\right>,\vspace{0.25cm}\\
\Gamma(\mathcal{F}_2)_{\mid U}=\left<\frac{\partial}{\partial
	q^{1}},\dots,\frac{\partial}{\partial q^{n}}\right>.\end{cases}
\end{equation*}
Such a chart, and such local coordinates, are said to be adapted to the bi-Lagrangian structure $(\mathcal{F}_{1},\mathcal{F}_{2})$. Moreover, if
 $$\omega=\sum_{i=1}^{n}dq^i\wedge dp^i,$$
then such a chart, and such local coordinates, are said to be adapted to the bi-Lagrangian structure $(\omega,\mathcal{F}_{1},\mathcal{F}_{2})$.


Let $\nabla $ be a linear connection. The torsion tensor $T_{\nabla}$ (or simply $T$ if there is no ambiguity) and curvature tensor $R_{\nabla}$  (or simply $R$) are given respectively by
$$ T_{\nabla}(X,Y)=\nabla_XY-\nabla_YX-[X,Y], \; X,Y\in\mathfrak{X}(M)$$
and
$$ R_{\nabla}(X,Y)Z=\nabla_X{\nabla_YZ}-\nabla_Y{\nabla_X^Z}-\nabla_{[X,Y]} , \; X,Y,Z\in\mathfrak{X}(M).$$

We say that a bi-Lagrangian structure (manifold) is affine when its Hess connection $\nabla$ is a curvature-free connection; that is, $\nabla$ is flat.
The set of affine bi-Lagrangian structure  is characterized  in Theorem~\ref{c10}.

We say that a connection $\nabla$
\begin{enumerate}
	\item[-] parallelizes $\omega$ if  $\nabla\omega=0$; this means,
	\begin{equation*} \label{Bieq3}
	\omega(\nabla_{X}{ Y},Z)+\omega(Y,\nabla_{ X}{Z})=X\omega(Y,Z),\; X,Y, Z\in\mathfrak{X}(M);
	\end{equation*}
	\item[-] preserves $\mathcal{F}$ if  $\nabla {\Gamma\left(\mathcal{F}\right)}\subseteq \Gamma\left(\mathcal{F}\right)$; more precisely,
	\begin{equation*}  \label{Bieq4}
	\nabla_XY\in\Gamma\left(\mathcal{F}\right),\; (X,Y)\in \mathfrak{X}(N)\times\Gamma\left(\mathcal{F}\right).
	\end{equation*}	
\end{enumerate}


Einstein summation convention: an index repeated as sub and superscript in a product represents summation over the range of the index. For example,
$$\lambda^j\xi_j=\sum_{j=1}^n \lambda^j\xi_j.$$
In the same way,
$$X^j\frac{\partial}{\partial y^j}=\sum_{j=1}^nX^j\frac{\partial}{\partial y^j}.$$

Let $k\in\mathbb{N}$.  Instead of $\{1,2,\dots,k\}$ we will simply write $[k]$.

\subsubsection{Prolongations of some objects to tangent bundle}
The prolongations  of tensor fields and linear connexion have been already defined by several authors, see \cite{SS, YK, SI}.
We are going to present the vertical and complete lifts of a tensor field, the vertical lift of a linear connection, and introduce the vertical lift of a foliation. If $T$ is a tensor field on a manifold $M$, $T^v$ and $T^c$ mean the vertical and the complete lifts of $T$ on $TM$ respectively. 

Let $\pi: TM\longrightarrow M$ be the natural projection,  $f\in C^{\infty}(M)$, $X\in\mathfrak{X}(M)$ and $\alpha\in\Omega^1(M)$ (the set of 1-forms on $M$).
\begin{enumerate}
\item[-] $f^v=f\circ \pi$ and $f^c=df$;
\item[-] $X^v(f^c)=(Xf)^v$ and $X^c(f^c)=(Xf)^c$;
\item[-] $\alpha^v(Y^c)=(\alpha(Y))^v$ and $\alpha^c(Y^c)=(\alpha(Y))^c$, $Y\in\mathfrak{X}(M)$. 
\end{enumerate}
From these definitions, it follows that 
\begin{equation}\label{eqlift1}
\begin{array}{ll}
(fg)^v=f^vg^v,& (fg)^c=f^cg^v+f^vg^c, \\
(fX)^v=f^vX^v,& (fX)^c=f^cX^v+f^vX^c,\\
(f\alpha)^v=f^v\alpha^v,& (f\alpha)^c=f^cg^v+f^v\alpha^c
\end{array}
\end{equation}
where $f,g\in C^{\infty}(M)$, $X\in\mathfrak{X}(M)$ and $\alpha\in\Omega^1(M)$.

By (\ref{eqlift1}), the vertical and complete lifts of a tensor field on $M$ are defined inductively by using the following formulas: 
\begin{equation}\label{eqlift2}
\begin{array}{ll}
(S+ T)^v=S^v+ T^v,&
 (S+ T)^c=S^c+ T^c,\\
(S\otimes T)^v=S^v\otimes T^v,& (S\otimes T)^c
\end{array}
\end{equation}
 for all tensor fields $S$ and $T$ on $M$. As a consequence, if $T$ is a $(r,0)$ tensor, then 
 \begin{equation}\label{eqlift3}
 T^c(X_1^c,\dots, X_r^c)=(T(X_1,\dots, X_r))^c\; \mbox{for all } X_1,\dots, X_r\in\mathfrak{X}(M).
 \end{equation}

In this paper, we just need the complete lifts of a linear connection and a foliation. The complete lift of a linear connection $\nabla$ is 
\begin{equation}\label{lift of connection}
\nabla^c_{X^c}Y^c=(\nabla_{ X}Y)^c,
\end{equation}
and if $T$ is a tensor, then
\begin{equation}\label{lct}
\nabla^cT=(\nabla T)^c.
\end{equation}
If $\mathcal{F}$ is a foliation, then $\mathcal{F}^c$ is the foliation whose $\Gamma\left(\mathcal{F}^c\right)$ the set of sections is  
\begin{equation}\label{lift of foliation}
\Gamma\left(\mathcal{F}^c\right):=\{X^c,\; X \in\Gamma\left(\mathcal{F}\right)\}=\left(\Gamma\left(\mathcal{F}\right)\right)^c.
\end{equation}


\subsubsection{Cherry vector field and map}
Let $X$ be a vector field on the 2-torus $\mathbb{T}^2$ and let $x$ be a point of $\mathbb{T}^2$.

 We denote by: 
\begin{enumerate}
	\item[-]$Sin(X)$ the set of singularities of $X$;
	\item[-] $t\mapsto \Phi^t_X$ the flow of $X$ on a neighborhood of  $x$. That is, there is a neighborhood $U$ of $x$ such that \begin{equation*}
	\left.\dfrac{d}{dt}\right | _{t=0}(\Phi^t_X(y))=X_y\; \mbox{ for every } y\in U;
	\end{equation*} 
		\item[-] $\gamma^-(x)=\{\Phi^t_X, \; t\leq 0\}$ and 
		$\gamma^+(x)=\{\Phi^t_X, \; t\geq 0\}$ the negative and positive semi-trajectories of $x$ respectively;
	\item[-]$\gamma(x)=\gamma^-(x)\cup\gamma^+(x)$ the one-dimensional trajectory of $x$. 
\end{enumerate}
Note that if $\gamma(x)$ and the circle $\mathbb{S}^1$ are homeomorphic, then $\gamma(x)$ is a closed trajectory or is a periodic trajectory.
the trajectory $\gamma(x)$ is said to be non-closed if it is neither a fixed point nor periodic trajectory. 

A vector field without closed trajectory is said to be Cherry on $\mathbb{T}^2$ if it has two singularities, a sink and a saddle, both hyperbolic. We denote by $\mathfrak{X}_c(\mathbb{T}^2)$ the set of Cherry vector field on $\mathbb{T}^2$.

The first example of such a vector field was given by Cherry, see \cite{TC}.

%

In this paper, We link pairs of vector fields with a class $\mathscr{L} $ of circle maps with a flat piece noted.

We consider  $\mathbb{S}^1$ as the interval $[0,1]$  where we identify
0 with 1.
We fix $\ell_1, \ell_2\geq0$, an interval $U=(a,b)\subset [0,1]$ and,   a low dimensional  order preserving map $f$  belongs to $\mathscr{L} $ with critical exponents $(\ell_1, \ell_2)$ and flat piece $U$ if the following holds:
\begin{enumerate} 
	\item The image of $U$ is one point.
	\item The restriction of $f$ to $[0,1]\setminus\overline{U} $ is a diffeomorphism onto its image.
	\item 
	On a left-sided neighborhood of $a$, $f$ equals
	\begin{equation*}
	h_l((x-a)^{\ell_1})
	\end{equation*}
	
	where $h_l$ is a diffeomorphism  on a two-sided neighbourhood of $a$.
	Analogously, on some right-sided neighborhood of $b$, $f$ can be represented as
	\begin{equation*}
	h_r((x-b)^{\ell_2}).
	\end{equation*}
\end{enumerate}
The interval $U$ is called the flat piece (or interval) of $f$.  
A map in $\mathscr{L}$ is called a Cherry map.

We say that a map $f\in \mathscr{L}$ with $(\ell_1,\ell_2)$ as its critical exponents is symmetric if $\ell_1=\ell_2$. 

Let $F$ be a lift of $f\in \mathscr{L}$ 
on the real line. The rotation number $\rho (f)$ of $f$ is  defined (independently of $F$)  by 
\begin{equation*}
\rho (f):=\lim_{n\rightarrow\infty }\dfrac{F^n(0)}{n}(mod\:1).
\end{equation*}
Note that a map $f$ in $\mathscr{L}$ has a periodic point if and only if $\rho (f)$ is a rational number. The class of maps in $\mathscr{L}$ with irrational rotation number is more interesting, see \cite{MSM, GJST, GJ, LP1, TNB1, TNB2}.




\subsection{Technical tools}
In this part, we present results that we will need in the following of this work.
\subsubsection{Symplectic manifolds}
Symplectic manifolds have been studied since 1780. Among many results on the theory, the cotangent bundle of a manifold is endowed with a so-called tautological 2-form. This part is devoted to the precise formulation of this result. For more familiarization with the concepts covered here, the reader is referred to \cite{dasilva, paul}.
%

Let $M$ be a $m$-manifold 
and let $^*\hspace{-0.1cm}\pi: TM\longrightarrow M$ be the natural projection.
The tautological 1-form or Liouville 1-form $\theta$ is defined by
$$ \theta_{(x,\alpha_x)}(v)=\alpha_x\left(T_{(x,\alpha_x)}{^*\hspace{-0.1cm}\pi}(v)\right),\;  (x,\alpha_x)\in T^{*}M,\,v\in T_{(x,\alpha_x)}T^*M, $$
and its exterior differential $d\theta$  is called  the canonical symplectic form or Liouville 2-form  on  the cotangent
bundle $T^{*}M$.	

Note that for any coordinate chart $(U, x^1,\dots,x^m)$ on $M$, with associated cotangent coordinate chart $(T^*U, x^1,\dots,x^m, \xi_1,\dots,\xi_m)$ we have
$$\theta=\sum_{1}^{m} \xi_idx_i,$$
and
$$d\theta=\sum_{1}^{m} d\xi_i\wedge dx_i.$$
\begin{proposition}\label{liouville}
	Let $M$ be a manifold. The cotangent bundle $T^{*}M$ of $M$ endowed with   the canonical symplectic form $d\theta$ is a symplectic manifold. 	
\end{proposition}

\subsubsection{Bi-Lagrangian (Para-K\"{a}hler) manifolds}

In this part, we briefly give some needed results concerning  Hess  affine bi-Lagrangian structures and the push forward of a bi-Lagrangian structure.
The  result characterizes affine bi-Lagrangian structures.
\begin{theorem}\cite[Theor. 2:, p. 159]{7} \label{c10}
	Let  $(\omega,\mathcal{F}_1,\mathcal{F}_2)$ be a bi-Lagrangian structure on a       $2n$-manifold $M$ with
	$\nabla$ as its Hess connection. Then the following assertions are equivalent.
 	\begin{enumerate}
		\item[a.] The connection $\nabla$ is flat.
		\item[b.] Each point of $M$ has a coordinate chart adapted to $(\omega,\mathcal{F}_1,\mathcal{F}_2)$.
	\end{enumerate}
\end{theorem}
The following result shows how to push forward a bi-Lagrangian structure.
\begin{lemma}\cite[Lem. 2.2., p. 6]{TNB3}\label{Bilem1}
	Let $(M,\omega,\mathcal{F}_1,\mathcal{F}_2)$ be a bi-Lagrangian manifold  with $\nabla$ as its Hess connection, and let $N$ be a manifold which is diffeomorphic to $M$. Then for any diffeomorphism $\psi : M\longrightarrow N$, the structure $(( \psi^{-1})^*\omega,\psi_*\mathcal{F}_1, \psi_*\mathcal{F}_2)$ is  bi-Lagrangian  on $N$.
\end{lemma}

\subsubsection{Cherry (vector) map}
Cherry vectors have been studied by several authors. Among the many results found on the theory, we have the following one found in
  \cite[\S 6.2, 6.3, p. 148-149]{LP1}.
\begin{proposition}\label{BiTprop1}
Let $X$  be a Cherry vector field on $\mathbb{T}^2\approx \mathbb{S}^1\times [0,1]$ with $t\mapsto\Phi_X^t$ as its flow at the neighborhood of a point $s\in \mathbb{T}^2 $, let $\Delta$ be the set of points $x\in \mathbb{S}^1\times\{0\}$ such that $\Phi_X^t(x)\in \mathbb{S}^1\times\{1\}$ for some  $t>0$, and let $t(x)$ be the minimal $t>0$ such that $\Phi_X^t(x)\in \mathbb{S}^1\times\{1\}$. Then the map 
\begin{equation*}
f:\Delta\longrightarrow \mathbb{S}^1, \; x\mapsto\Phi_X^{t(x)}(x)
\end{equation*} 
belongs to $\mathscr{L}$. 
\end{proposition}




\section{Statements and proofs of results}

\subsection{Statements of  results}
Our first result presents lifted bi-Lagrangian structures on the trivial bundles of some manifolds.
\begin{theorem}\label{Bitheo1}
	Let $(M, \omega,\mathcal{F}_1,\mathcal{F}_2)$ be a bi-Lagrangian with $\nabla$ as its Hess connection. We have:
	\begin{enumerate}
	\item the quadruplet $(T^*M, d\theta,  N^*\mathcal{F}_1,N^*\mathcal{F}_2)$ is an affine bi-Lagrangian manifold; 
	\item  the expression $(T^*M, \tilde{\omega}:=^*\hspace{-0.1cm}\pi^*\omega+d\theta,  N^*\mathcal{F}_1,N^*\mathcal{F}_2)$ is a bi-Lagrangian manifold. 
	\item the object $(TM, \omega^c,\mathcal{F}_1^c,\mathcal{F}_2^c)$ is a bi-Lagrangian manifold with $\nabla^c$ as its Hess connection. Moreover, if $(M, \omega,\mathcal{F}_1,\mathcal{F}_2)$ is affine, then so is  $(M, \omega^c,\mathcal{F}_1^c,\mathcal{F}_2^c)$. 
	\end{enumerate}
\end{theorem}

\begin{remark}\label{Birem4}	
Let $(M, \omega)$ be a symplectic manifold. Since $TM$ is  diffeomorphic to  $T^*M$  (via the map $X\mapsto \omega(X,\cdot)$), then by Lemma~\ref{Bilem1}, a bi-Lagrangian structure on $TM$ induces a bi-Lagrangian structure on   $T^*M$, and vice-versa.	
\end{remark}

\begin{corollary}
	Let $M$ be a manifold equipped with a bi-Lagrangian structure. The action $(\psi,(\mathcal{F}_1,\mathcal{F}_2) ) \mapsto (\psi_*\mathcal{F}_1,\psi_*\mathcal{F}_2)$ of the symplectomorphism group  of $(M, \omega)$ on the set of its bi-Lagrangian structures defined 	in \cite[The. 2.3., p. 6]{TNB3} can be lifted on $TM$ and $T^*M$ in different ways.
\end{corollary}

\begin{theorem}~\label{Bitheo2}
\begin{enumerate}
\item Some pairs of Cherry vector fields generate maps belonging in $\mathscr{L}$. 
\item The left action \begin{equation}\label{Biact1}
*:Diff(\mathbb{T}^2)\times\mathfrak{X}_c(\mathbb{T}^2)\longrightarrow \mathfrak{X}_c(\mathbb{T}^2),\;(\varphi, X)\mapsto \varphi_*X\tag{act1}\end{equation} induces the following conjugation action \begin{equation}\label{Biact2}
\circ: Diff(\mathbb{S}^1)\times\mathscr{L}\longrightarrow\mathscr{L},\;(\varphi, f)\mapsto \varphi\circ f\circ\varphi^{-1}.\tag{act2}\end{equation}
Moreover, for every $X\in\mathfrak{X}_c(\mathbb{T}^2)$ and for every $\varphi\in Diff(\mathbb{T}^2)$,  if $X$ generates the map $f$, then $\varphi_*X$ generates the map
\begin{equation*}
\varphi\circ f\circ\varphi^{-1}: \varphi(\Delta)\longrightarrow\mathbb{S}^1,\; x\mapsto\varphi\circ \Phi_X^{t(x)}\circ\varphi^{-1}(x)
\end{equation*}
where  $\Delta$ is the set of points $x\in \mathbb{S}^1\times\{0\}$ such that $\Phi_X^t(x)\in \mathbb{S}^1\times\{1\}$ for some  $t>0$, and  $t(x)$ is the minimal $t>0$ such that $\Phi_X^t(\varphi^{-1}(x))\in \mathbb{S}^1\times\{1\}$.

\end{enumerate}
\end{theorem}
The point 1  of this result extends results in
\cite[Theo. B., p. 533]{MSM} and \cite[\S 6.2, 6.3, p. 148-149]{LP1}. 
Now, we are going to characterize the orbit of a Cherry vector field with respect to $*$ (see,  
(\ref{Biact1})) from the orbit via $\circ$ (see, (\ref{Biact2})) of a map in $\mathscr{L}$. For every $X\in\mathfrak{X}_c(\mathbb{T}^2)$ we denote by $f^X$ the map in $\mathscr{L}$ generated by $X$. We have the following observation.
%
\begin{remark}
The orbit with respect to $\circ$ (see, (\ref{Biact2})) of a map $f\in \mathscr{L}$ noted $\mathcal{O}_{\circ}(f)$ is defined by
\begin{equation*}
\mathcal{O}_{\circ}(f):=\{\varphi\circ f\circ\varphi^{-1},\; \varphi \mbox{ is homeomorphism}  \}.
\end{equation*}
Since two maps belonging in $\mathscr{L}$ with the same irrational rotation number belong in the same orbit (see, \cite[Theo. D., p. 535]{MSM}), then 
\begin{equation*}
\mathcal{O}_{\circ}(f):=\{g\in \mathscr{L}:\; \rho(f) =\rho(g)   \}.
\end{equation*}
As a consequence, for every $X\in\mathfrak{X}_c(\mathbb{T}^2)$ such that $\rho(f^X)$ is a irrational number, we have
\begin{equation*}
\mathcal{O}_{*}(X):=\{Y\in\mathfrak{X}_c(\mathbb{T}^2):\; \rho\left(f^X\right) =\rho\left(f^Y\right)   \}.
\end{equation*}
Note that these observations still hold  for $C^2$ Cherry vector field (or Cherry map).
\end{remark}





\subsection{Proofs of results}

\subsubsection{Lifted bi-Lagrangian structures }

\begin{proof}[Proof of Theorem~\ref{Bitheo1}]
	\begin{lemma}\label{Bilem2}
		Let $\mathcal{F}=\{S_x\}_{x\in M}$ be a $k$ foliation on a $m$ manifold $M$. Then $N^*\mathcal{F}$ is a  Lagrangian foliation on $(T^*M, d\theta)$. Moreover, if $M$ is endowed with a symplectic form $\omega$ and  $\mathcal{F}$ is Lagrangian on $(M,\omega)$,  so is $N^*\mathcal{F}$ on $(M,\omega)$.
		
	\end{lemma}
\begin{proof}
First note that if $S\subset M$ is a sub-manifold, then $N^*S$ is a Lagrangian sub-manifold of $(T^*M, d\theta)$ (this is a general result, see \cite[Cor. 3.7., p. 18]{dasilva}). As a consequence, $N^*\mathcal{F}=\{N^*S_x\}_{x\in M}$ is Lagrangian on $(T^*M, d\theta)$.
Thus it remains to show that $N^*\mathcal{F}$ is completely integrable. 
This means
\begin{equation*}
d\theta([X,Y], Z)=0\; \mbox{ for all } X,Y,Z\in\Gamma(N^*\mathcal{F})
\end{equation*}
 since $N^*\mathcal{F}$ is Lagrangian.
  
 	Note that
 \begin{equation*}
 d\theta([X,Y],Z)=[X,Y]\theta(Z)-Z\theta([X,Y])-\theta([[X,Y],Z]).
 \end{equation*}
 
 Let $(U, p^1,\dots, p^m)$ be a  coordinate system chart adapted to the foliation $\mathcal{F}$  with $(T^*U, p^1,\dots,p^{m}, \xi_1,\dots,\xi_{m})$ as its associated bundle coordinate chart. Observe that 
 \begin{equation}\label{Bieq8}
 \Gamma(N^*\mathcal{F})|_{T^*U}=\left<\frac{\partial}{\partial
 	p^1},\dots,\frac{\partial}{\partial   p^k},\frac{\partial}{\partial
 	\xi_{k+1}},\dots,\frac{\partial}{\partial
 	\xi_m}\right>
 \end{equation}
 
 Let us write
 \begin{equation*}
 \begin{cases}
 (y^i)_{i=1,\dots,m}=((p^i)_{i=1,\dots,m},(\xi_{i})_{i=k+1,\dots,m}),\\
 X=X^i\frac{\partial}{\partial y^i},Y=Y^j\frac{\partial}{\partial
 	y^j}\;\mbox{ and }\;Z=Z^k\frac{\partial}{\partial
 	y^k}.\end{cases}\end{equation*} Then \begin{equation*}
 \begin{cases}
 [X,Y]=\mu^j\frac{\partial}{\partial y^j},\\
 [[X,Y],Z]=\lambda^j\frac{\partial}{\partial
 	y^j},\end{cases}\end{equation*} where
 \begin{equation*}
 \begin{cases}
 \mu^j=X^i\frac{\partial Y^j}{\partial y^i}-Y^i\frac{\partial
 	X^j}{\partial y^i},\\
 \lambda^j=\mu^i\frac{\partial Z^j}{\partial y^i}-Z^i\frac{\partial
 	\mu^j}{\partial y^i}.\end{cases}\end{equation*} Thus,
 \begin{align}
 [X,Y]\theta(Z)&=\mu^i\frac{\partial}{\partial
 	y^i}(Z^k\xi_k),\tag{$e_1$}\label{Bieq5}\\
 \theta([[X,Y],Z])&=\lambda^j\xi_j, \tag{$e_2$}\label{Bieq6}
 \\Z\theta([X,Y])&=Z^k\frac{\partial}{\partial y^k}(\mu^i\xi_i). \tag{$e_3$}\label{Bieq7}
 \end{align}
 Therefore
 \begin{equation*}
 d\theta([X,Y],Z)=\mbox{(\ref{Bieq5})}-\mbox{(\ref{Bieq6})}-\mbox{(\ref{Bieq7})}=0.
 \end{equation*}
	
So $N^*\mathcal{F}$ is a  Lagrangian foliation on $(T^*M, d\theta)$.

Now, suppose that $\mathcal{F}$ is a Lagrangian foliation on $(M, \omega)$. Observe that for every  $X\in\Gamma(N^*\mathcal{F})$, $^*\hspace{-0.1cm}\pi_*X\in\Gamma(\mathcal{F})$. Thus for all $X, Y\in\Gamma(N^*\mathcal{F})$, 
\begin{equation*}
\tilde{\omega}(X,Y)=\omega(^*\hspace{-0.1cm}\pi_*X,^*\hspace{-0.1cm}\pi_*Y )+d\theta(X,Y)=0
\end{equation*}
where we use the fact that $\mathcal{F}$ is Lagrangian and the definition of $d\theta$.
So $N^*\mathcal{F}$ is a Lagrangian foliation on $(T^*M,\tilde{\omega} )$. This ends the proof of Lemma~\ref{Bilem2}.
\end{proof}
Now we are ready to prove Theorem~\ref{Bitheo1}. 	Let $(M, \omega,\mathcal{F}_1,\mathcal{F}_2)$ be a bi-Lagrangian manifold.
\begin{enumerate}
\item By Lemma~\ref{Bilem2}, $(N^*\mathcal{F}_1,N^*\mathcal{F}_2)$ is a pair of Lagrangian foliation on $(T^*M, d\theta)$. 
From (\ref{Bieq8}), it follows that if $(U, p^1,\dots,p^{n},q^1,\dots,q^{n})$ is  a  coordinate chart adapted to the bi-Lagrangian structure $(\mathcal{F}_1,\mathcal{F}_2)$, with $(T^*U, p^1,\dots,p^{n},q^1,\dots,q^{n}, \xi_1,\dots,\xi_{2n})$ as its associated bundle coordinate chart,  then
\begin{equation}\label{Bieq9}
\begin{cases}
\Gamma(N^*\mathcal{F}_1)|_{T^*U}=\left<\frac{\partial}{\partial
	p^1},\dots,\frac{\partial}{\partial   p^n},\frac{\partial}{\partial
	\xi_{n+1}},\dots,\frac{\partial}{\partial \xi_{2n}}\right>,\vspace{0.25cm}\\
\Gamma(N^*\mathcal{F}_2)|_{T^*U}=\left<\frac{\partial}{\partial
	q^{1}},\dots,\frac{\partial}{\partial q^{n}},\frac{\partial}{\partial
	\xi_1},\dots,\frac{\partial}{\partial
	\xi_n}  \right>\end{cases}
\end{equation}
and 
 \begin{equation}\label{Bieq10}
 d\theta|_{T^*U}=\sum_{i=1}^{n}(d\xi_i\wedge dp^i+d\xi_{n+i}\wedge dq^i).
 \end{equation}
 
By (\ref{Bieq9}), we get that  $N^*\mathcal{F}_1$  and $N^*\mathcal{F}_2)$ are transverse. so $(N^*\mathcal{F}_1,N^*\mathcal{F}_2)$ is a bi-Lagrangian structure on $(T^*M, d\theta)$. By combining equalities in (\ref{Bieq9}), and  equality (\ref{Bieq10}), it follows from Theorem~\ref{c10} that $(N^*\mathcal{F}_1,N^*\mathcal{F}_2)$ is affine. 
\item The point 2 of Theorem~\ref{Bitheo1} follows directly by combining Lemma~\ref{Bilem2} and equalities in (\ref{Bieq9}).
\item  By (\ref{eqlift3}), we have 
\begin{equation*}
\omega^c(X^c,Y^c)=(\omega(X,Y))^c\; \mbox{ for all } X,Y\in\mathfrak{X}(M).
\end{equation*}
Thus, since the structure $(\mathcal{F}_1,\mathcal{F}_2)$ is bi-Lagrangian on $(M,\omega)$, then  $(\mathcal{F}_1^c ,\mathcal{F}_2^c)$ is a pair of transversal Lagrangian foliation on $(TM, \omega^c)$. 

Let $\nabla$ be the Hess connection of $(\omega,\mathcal{F}_1, \mathcal{F}_2)$. We are going to show that  $\nabla^c$ is that of  $(\omega^c,\mathcal{F}_1^c, \mathcal{F}_2^c)$. 

By combining the fact that  $\nabla$ is the Hess connection of $(\omega,\mathcal{F}_1, \mathcal{F}_2)$, (\ref{eqlift3}), (\ref{lct}) and (\ref{lift of connection}) we have
\begin{equation}\label{eqHc0}
T^c(X^c,Y^c)=T^c(X^c,Y^c)=0
\end{equation}
\begin{equation}\label{eqHc1}
\nabla^c\omega^c=(\nabla\omega)^c=0
\end{equation}
\begin{equation}\label{eqHc2}
\nabla_{X^c}Y^c=(\nabla_{X}Y)^c\in\Gamma(\mathcal{F}^c_i)\; X,Y\in \Gamma(\mathcal{F}i)\; i=1,2.
\end{equation}
By combining (\ref{eqHc0}), (\ref{eqHc1}) and (\ref{eqHc2}), it follows that $\nabla^c$ is the Hess connection of $(\omega^c,\mathcal{F}_1^c, \mathcal{F}_2^c)$. This ends the proof of Theorem~\ref{Bitheo1}.  
\end{enumerate}
\end{proof}

\begin{proof}[Theorem~\ref{Bitheo2}]
\begin{enumerate}
\item Let $X_1$ $X_2$ be two Cherry vector field such that the following holds: 
\begin{enumerate}
\item[-] the vector field $X_1$ generates $f_1\in\mathscr{L}$ with $U_1=(a_1,b_1)$ as its flat piece, and $(\ell_1,\ell_1)$ as its critical exponents,
\item[-] the vector field $X_2$ generates $f_1\in\mathscr{L}$ with $U_2=(a_2,b_2)$ as its flat piece, and $(\ell_2,\ell_2)$ as its critical exponents, 
\item[-] $f_1(U_1)=f_2(U_2)$,
\item[-] the intervals $(a_1,b_1)$ and $(a_2,b_2)$ are such that $a_1<a_2\leq b_1<b_2$.
\end{enumerate}
The map 
\begin{equation*}
f(x)=\begin{cases}\begin{array}{lll}
f_1(x) & \mbox{if } & x\in [0,a_1]\\
f_2(x) & \mbox{if } & x\in [a_2,1]
\end{array}
\end{cases}
\end{equation*}
belongs to $\mathscr{L}$ with flat piece $U=(a_1,b_2)$ and critical exponents  $(\ell_1,\ell_2)$.
\item Let $X\in \mathbb{T}^2$  with $\Phi_X:\mapsto\Phi_X^t$ as its flow on an open $U$. Let $\varphi$ be a diffeomorphism. 
observe that 
\begin{equation*}
\left.\dfrac{d}{dt}\right | _{t=0}(\varphi\circ\Phi^t_X\circ\varphi^{-1}(y))=(\varphi_*X)_y\; \mbox{ for every } y\in U.
\end{equation*} 
This means $\varphi\circ\Phi^t_X\circ\varphi^{-1}$ is the flow of $\varphi_*X$.
This ends the proof of Theorem~\ref{Bitheo2}.
\end{enumerate}
\end{proof}

%

\end{document}